%% file: paper.tex
\documentclass[final]{siamltex}

\newcommand{\R}{\mathbb{R}}

\newcommand{\dx}{\, \mathrm{d}x}
\newcommand{\dX}{\, \mathrm{d}X}
\newtheorem{example}{Example}
\newtheorem{remark}{Remark}

\usepackage{amsmath}
\usepackage{graphicx}
\usepackage{amssymb}
\usepackage{url}
\usepackage{algorithm}
\usepackage{algorithmic}

\title{Topological optimization of the evaluation of finite element matrices}

\author{Robert~C.~Kirby
\thanks
{Department of Computer Science, University of
Chicago, Chicago,
Illinois 60637-1581, USA;
This work was supported by the U.S. Department of Energy Early Career
Principal Investigator Program under award number DE-FG02-04ER25650.
}
\and Anders~Logg
\thanks{Toyota Technological Institute at Chicago,
 University Press Building, 1427 East 60th Street, Chicago, Illinois
 60637, USA.}
\and L.~Ridgway~Scott
\thanks{The Computation Institute and
Departments of Computer Science and Mathematics, University of
Chicago, Chicago,
Illinois 60637-1581, USA.}
\and Andy~R.~Terrel
\thanks{Department of Computer Science, University of
Chicago, Chicago,
Illinois 60637-1581, USA
}
}
\begin{document}

\maketitle

\input{sections/abstract.tex}

\begin{keywords}
finite element, variational form, optimized algorithm, minimum
spanning tree
\end{keywords}

\begin{AMS}
65N30, 05C90, 68R10
\end{AMS}

\pagestyle{myheadings}
\thispagestyle{plain}
\markboth{R. C. KIRBY AND A. LOGG AND L. R. SCOTT AND
  A. R. TERREL}{OPTIMIZING EVALUATION OF FINITE
  ELEMENT MATRICES}

\input{sections/intro.tex}
\input{sections/feform.tex}
\input{sections/complexityreducing.tex}
\input{sections/experimental.tex}
\input{sections/search.tex}
\input{sections/conclusion.tex}

\bibliographystyle{siam}
\bibliography{bibliography}

\end{document}

%% file: sections/abstract.tex
\begin{abstract}
We present a topological framework for finding low-flop algorithms for
evaluating element stiffness matrices associated with multilinear
forms for finite element methods posed over straight-sided affine
domains.  This framework relies on phrasing the computation on each
element as the contraction of each collection of reference element
tensors with an element-specific geometric tensor.  We then present a
new concept of {\em complexity-reducing relations} that serve as
distance relations between these reference element tensors.  This
notion sets up a graph-theoretic context in which we may find an
optimized algorithm by computing a minimum spanning tree. We present
experimental results for some common multilinear forms showing
significant reductions in operation count and also discuss some
efficient algorithms for building the graph we use for the
optimization.
\end{abstract}

%% file: sections/intro.tex
\section{Introduction}

Several ongoing projects have led to the development of tools for
automating important aspects of the finite element method, with the
potential for increasing code reliability and decreasing development
time. By developing libraries for existing languages or new
domain-specific languages, these software tools allow programmers to
define variational forms and other parts of a finite element method
with succinct, mathematical syntax.  Existing C++ libraries for finite
elements include DOLFIN~\cite{logg:www:01,logg:preprint:06},
Sundance~\cite{www:Sundance,Lon03} and deal.II~\cite{www:deal}.  These
rely on some combination of operator overloading and object
orientation to present a high-level syntax to the user.  Other
projects have defined new languages with their own grammar and syntax,
such as FreeFEM~\cite{www:FreeFEM}, GetDP~\cite{www:GetDP}, and
Analysa~\cite{www:Analysa}.  Somewhere between these two approaches is
the FEniCS Form Compiler, FFC~\cite{logg:submit:04,logg:www:04}
developed primarily by the second author.  This Python library
relies on operator overloading to define variational forms, but rather
than using on the Python interpreter to evaluate the forms, FFC
generates low-level code that can be compiled into other platforms such
as DOLFIN.

While these tools are effective at exploiting modern software
engineering to produce workable systems, we believe that additional
mathematical insight will lead to even more powerful codes with more
general approximating spaces and more powerful algorithms.  To give
one example, work by the first author~\cite{Kir04,Kir05} shows how the Ciarl\'et
definition of the finite element leads to a code for arbitrary order
elements of general type.  This code, FIAT, is used by FFC and is
currently being integrated with Sundance.  For further examples, we refer to
work by the first three authors and Knepley~\cite{logg:article:07,KirKne05}, which studies
how to efficiently (in the sense of operation count) evaluate local
stiffness matrices for finite element methods.  For a given variational form, for each element $e$, there exists a "geometric tensor" $G_e$.  Each entry of the $n \times n$ element stiffness matrix for $e$ may be written as the contraction of a specific reference element tensor (depending only on the form, the polynomial basis and the reference element) with $G_e$.
We saw that there are relations between many of the reference element tensors
(equality, colinearity, small Hamming distance) that may be exploited to find an
algorithm with significantly fewer floating point operations than standard techniques.

In ~\cite{logg:article:07}, we devised a crude algorithm that searched
for and exploited these dependencies, generating simple Python code
for evaluating the local stiffness matrix for the Laplacian on
triangles given the geometric tensor.  In this paper, we set the
optimization process in a more abstract graph/topological context.  In
Section~\ref{feform}, we express the calculation for some finite
element operators as tensor contractions.  These ideas are more fully
developed in~\cite{logg:submit:04}, where the first two authors show
how this tensor structure may be used to implement a compiler capable
of generating the reference tensors and hence code for building the
stiffness matrices for general multilinear forms on affine elements.  Further,
a generalization to non-affine mappings, including curvilinear elements, is suggested.
In Section~\ref{optimize}, we introduce the idea of {\em complexity-reducing
relations}, which are a kind of distance relations among the tensors
that serve to model the cost of complexity in computing the
contractions.  Using these ideas, we show how to derive an algorithm
for performing the computation that is optimal in a certain sense.  We
demonstrate the reduction in operation count for several finite
element operations in Section~\ref{experiment}.  After this, we show
how these efficient algorithms may be derived more efficiently through
the use of certain search procedures in Section~\ref{efficient}.  Finally, we state some
conclusions and remarks about ongoing work in Section~\ref{conc}.

%% file: sections/feform.tex
\section{Finite element formulation}
\label{feform}

The finite element method is a general methodology for the
discretization of differential equations. A linear (or linearized)
differential equation for the unknown function $u$ is expressed in the
form of a canonical variational problem and the discrete approximation
$U$ of $u$ is sought as the solution of a discrete version of the
variational problem \cite{Cia76,BreSco94}: Find $U \in V$ such that
\begin{equation} \label{eq:varform}
  a(v, U) = L(v) \quad \forall v \in \hat{V},
\end{equation}
where $(\hat{V}, V)$ is a pair of suitable discrete (typically
piecewise polynomial) function spaces, $a : \hat{V} \times V
\rightarrow \R$ a bilinear form and $L : \hat{V} \rightarrow \R$
a linear form.

The variational problem (\ref{eq:varform}) corresponds to a linear
system $A \xi = b$ for the expansion coefficients $\xi \in \R^M$ of
the discrete function $U$ in a basis $\{\varphi_i\}_{i=1}^M$ for
$V$. If $\{\hat{\varphi}_i\}_{i=1}^M$ is the corresponding basis for
$\hat{V}$, the entries of $A$ and $b$ are given by $A_{ij} =
a(\hat{\varphi}_i, \varphi_j)$ and $b_i = L(\hat{\varphi_i})$
respectively. When we consider general multilinear forms
below, the multilinear form $a$ is represented by a \emph{tensor} $A$.

\subsection{Multilinear forms}

Let now $\{V_i\}_{i=1}^r$ be a given set of discrete function spaces defined
on a triangulation $\mathcal{T} = \{e\}$ of $\Omega \subset \R^d$.
We consider a general multilinear form $a$ defined on the product space
$V_1 \times V_2 \times \cdots \times V_r$:
\begin{equation}
  a : V_1 \times V_2 \times \cdots \times V_r \rightarrow \R.
\end{equation}
Typically, $r = 1$ (linear form) or $r = 2$ (bilinear form), but
forms of higher arity appear frequently in applications and include
variable coefficient diffusion and advection of momentum in the
incompressible Navier--Stokes equations.

Let
$\{\varphi_i^1\}_{i=1}^{M_1},
 \{\varphi_i^2\}_{i=1}^{M_2}, \ldots,
 \{\varphi_i^r\}_{i=1}^{M_r}$
be bases of $V_1, V_2, \ldots, V_r$ and let $i = (i_1, i_2, \ldots,
i_r)$ be a multiindex. The multilinear form $a$ then
defines a rank $r$ tensor given by
\begin{equation}
  A_i = a(\varphi_{i_1}^1, \varphi_{i_2}^2, \ldots, \varphi_{i_r}^r).
\end{equation}
In the case of a bilinear form, the tensor $A$ is a matrix (the
stiffness matrix), and in the case of a linear form, the tensor $A$ is
a vector (the load vector).

To compute the tensor $A$ by assembly, we need to compute the
\emph{element tensor} $A^e$ on each element $e$ of the triangulation
$\mathcal{T}$ of $\Omega$ \cite{logg:submit:04}. Let
$\{\varphi^{e,1}_i\}_{i=1}^{n_1}$ be the restriction to $e$ of the
subset of $\{\varphi_i^1\}_{i=1}^{M_1}$ supported on $e$ and define
the local bases on $e$ for $V_2, \ldots, V_r$ similarly. The rank $r$
element tensor $A^e$ is then given by
\begin{equation}
  A^e_i = a_e(\varphi_{i_1}^{e,1}, \varphi_{i_2}^{e,2}, \ldots, \varphi_{i_r}^{e,r}).
\end{equation}

\subsection{Evaluation by tensor representation}
\label{sec:representation}

The element tensor $A^e$ can be efficiently computed by representing
$A^e$ as a special tensor product. If the multilinear form~$a$ is
given by an integral over the domain $\Omega$, then each entry
$A^e_i$ of $A^e$ is given by an integral over the element
$e$. By a change of variables and a series of linear transformations,
we may rewrite this integral as an integral over a reference element
$E$. In particular, when the map from the reference element $E$ is
affine, the linear transformations of derivatives can be moved outside
of the integral to obtain a representation of the element tensor $A^e$
as a tensor product of a constant tensor $A^0$ and a tensor $G_e$ that
varies over the set of elements,
\begin{equation} \label{eq:tensorproduct}
  A^e_{i} = A^0_{i\alpha} G_e^{\alpha},
\end{equation}
or more generally a sum $A^e_{i} = A^{0,k}_{i\alpha} G_{e,k}^{\alpha}$
of such tensor products, where $i$ and $\alpha$ are multiindices and
we use the convention that repetition of an index means summation over
that index. We refer to $A^0$ as the \emph{reference tensor} and to
$G_e$ as the \emph{geometric tensor}.  The rank of the reference tensor
is the sum of the rank $r = |i|$ of the element tensor and the rank
$|\alpha|$ of the geometric tensor $G_e$. As we shall see, the rank of
the geometric tensor depends on the specific form.

In \cite{logg:submit:04}, we present an algorithm that computes the
tensor representation (\ref{eq:tensorproduct}) for fairly general
multilinear forms. This algorithm forms the foundation for the FEniCS
Form Compiler, FFC \cite{logg:www:04}.

As an example, we consider here the tensor representation of the
element tensor $A^e$ for Poisson's equation $-\Delta u(x) = f(x)$ with
homogeneous Dirichlet boundary conditions on a domain $\Omega$.  The
bilinear form $a$ is here given by $a(v, u) = \int_{\Omega} \nabla
v(x) \cdot \nabla u(x) \dx$ and the linear form $L$ is given by $L(v)
= \int_{\Omega} v(x) f(x) \dx$. By a change of variables using an
affine map $F_e : E \rightarrow e$, we obtain the following
representation of the element tensor $A^e$:
\begin{equation}
  \begin{split}
    A^e_i &= \int_{e} \nabla \varphi^{e,1}_{i_1}(x) \cdot \nabla \varphi^{e,2}_{i_2}(x) \dx \\
    &= \det F_e' 
    \frac{\partial X_{\alpha_1}}{\partial x_{\beta}}
    \frac{\partial X_{\alpha_2}}{\partial x_{\beta}}
    \int_{E}
    \frac{\partial \Phi^1_{i_1}(X)}{\partial X_{\alpha_1}}
    \frac{\partial \Phi^2_{i_2}(X)}{\partial X_{\alpha_2}} \dX
    = A^0_{i\alpha} G_e^{\alpha},
  \end{split}
\label{eq:laptensorrep}
\end{equation}
where
$A^0_{i\alpha} = \int_{E}
\frac{\partial \Phi^1_{i_1}(X)}{\partial X_{\alpha_1}}
\frac{\partial \Phi^2_{i_2}(X)}{\partial X_{\alpha_2}} \dX$,
$G_e^{\alpha} = \det F_e'
\frac{\partial X_{\alpha_1}}{\partial x_{\beta}}
\frac{\partial X_{\alpha_2}}{\partial x_{\beta}}$,
and $\Phi^j_{i_j} = \varphi^{e,j}_{i_j} \circ F_e$ for $j = 1,2$.
We see that the reference tensor $A^0$ is here a rank four tensor and
the geometric tensor $G_e$ is a rank two tensor.
In~\cite{logg:article:07}, we saw that dependencies frequently occur
between $A^0_i$ and $A^0_j$ for many multiindices $i,j$ that can reduce the overall
cost of computation.

%% file: sections/complexityreducing.tex
\section{Optimizing stiffness matrix evaluation}
\label{optimize}

\subsection{An abstract optimization problem}
Tensors of arbitrary rank may be rewritten as vectors, with contraction
becoming the Euclidean inner product.  So, we will just deal with vectors
and inner products to formulate our abstract optimization problem.
We let $Y = \{y^i \}_{i=1}^n$ be a collection of $n$ vectors in $\R^m$. In the most
general case, this is a collection rather than a set, as some of the
items in $Y$ may be identical.  Corresponding to $Y$, we must find a
process for computing for arbitrary $g \in \R^m$ the collection of
items $ \{ \left(y^i\right)^t g \}_{i=1}^n$.  Throughout this paper,
we will measure the cost of this as the total number of multiply-add
pairs (MAPs) required to complete all the dot products.  This cost is
always bounded by $nm$, but we hope to improve on that. This could be
alternatively formalized as building an abstract control-flow graph
for performing the dot products that is equivalent to the naive
process but contains a minimal number of nodes.  Our techniques,
however, rely on structure that is not apparent to traditional
optimizing compilers, so we prefer the present formulation.

We seek out ways of optimizing the local matrix evaluation that rely
on notions of distance between a few of the underlying vectors.  The
Euclidean metric is not helpful here; we focus on other, discrete
measures of distance such that if $y$ and $z$ are close together, then
$y^t g$ is easy to compute once $z^t g$ is known (and vice versa).
Many of the dependencies we considered in~\cite{logg:article:07} were
between pairs of vectors --- equality, colinearity, and Hamming
distance.  Here, we develop a theory for optimizing the evaluation of
finite element matrices under binary relations between the members of
the collection.  This starts by introducing some notions of distance
on the collection of vectors and finds an optimized computation with
respect to those notions by means of a minimum spanning tree.

\subsection{Complexity-reducing relations}
\label{complexityreducing}

\begin{definition}
Let $\rho : Y \times Y \rightarrow [0,m]$ be symmetric.  We say that
$\rho$ is {\em complexity-reducing} if for every $y,z \in Y$ with
$\rho(y,z) \leq k < m$, $y^t g$ may be computed using the result $z^t g$
in no more than $k$ multiply-add pairs.
\end{definition}

\begin{example}
Let $e^+(y,z) = d (1 - \delta_{y,z})$, where $\delta_{y,z}$ is the
standard Kronecker delta.  Then, $e^+$ is seen to be
complexity-reducing, for if $e^+(y,z) = 0$, then $y^t g = z^t g$ for
all $g \in \R^m$ and hence the former requires no arithmetic once the
latter is known.  Similarly, we can let $e^-(u,v) = e^+(u,-v)$, for if
$u = -v$, then computing $u^t g$ from $v^t g$ requires only a sign
flip and no further floating point operations.
\end{example}

\begin{example}
Let
\begin{equation}
c(y,z) = \left\{
\begin{array}{ll}
0, & y = z \\
1, & y = \alpha z \, \mathrm{for \, some} \, \alpha \in \R, \alpha \neq 0,1\\
m, & {\text otherwise}
\end{array}
\right.
\end{equation}
Then $c$ is
complexity-reducing, for $y^t g = (\alpha z)^t g = \alpha ( z^t g )$,
so $y^t g$ may be computed with one additional floating point
operation once $z^t g$ is known.
\label{ex:colinear}
\end{example}

\begin{example}
Let $H^+(y,z)$ be the Hamming distance, the number entries in which
$y$ and $z$ differ. Then $H^+$ is complexity-reducing.  If
$H^+(y,z) = k$, then $y$ and $z$ differ in $k$ entries, so the
difference $y-z$ has only $k$ nonzero entries.  Hence, $(y-z)^t g$
costs $k$ multiply-add pairs to compute, and we may write $y^t g =
(y-z)^t g + z^t g$.  By the same argument, we can let $H^-(y,z) =
H^+(y,-z).$
\label{ex:hamming}
\end{example}

\begin{theorem}
Let $\rho_1$ and $\rho_2$ be complexity-reducing relations.  Define
\begin{equation}
\rho(y,z) = \min(\rho_1(y,z),\rho_2(y,z)).
\end{equation}
Then $\rho$ is a complexity-reducing relation.
\end{theorem}
\begin{proof}
Pick $y,z \in Y$, let $1 \leq i \leq 2$ be such that $\rho(y,z) =
\rho_i(y,z)$ and let $\rho_i(y,z) \equiv k$.  But $\rho_i$ is a
complexity-reducing relation, so for any $g \in \R^m$, $y^t g$ may be
computed in no more than $k = \rho(y,z)$ multiply-add pairs.  Hence
$\rho$ is complexity-reducing.
\end{proof}

This simple result means that we may consider any finite collection of
complexity-reducing relations (e.g. colinearity together with Hamming
distance) as if they were a single relation for the purpose of finding
an optimized computation in the later discussion.

\begin{definition}
If $\rho$ is a complexity-reducing relation defined as the minimum
over a finite set of complexity-reducing relations, we say that it is
{\em composite}.  If it is not, we say that it is {\em simple}.
\end{definition}

To see that not all complexity-reducing relations are metrics, it is
easy to find a composite complexity-reducing that violates the triangle inequality.
Let $\rho(y,z) = \min( H^+(y,z) , c(y,z) )$.  It is not hard to show that
$H^+$ and $c$ are both metrics.  If we take $y = (1,2,2)^t$ and
$z=(0,4,4)^t$, then $\rho(y,z) = 3$ since the vectors are neither
colinear nor share any common entries.  However, if we let $x =
(0,2,2)^t$, then $\rho(y,x) = 1$ since the vectors share two entries
and $\rho(x,z) = 1$ by colinearity.  Hence, $\rho(y,z) > \rho(y,x) +
\rho(x,z)$, and the triangle inequality fails.

\begin{remark}
It may be possible to enrich the collection of vectors with {\em Steiner points} to reduce the overall cost of the computation.  For example, if the collection of vectors $\{ (1,2,2)^t , (0,4,4)^t \}$ above were enriched with the additional vector $(0,2,2)^t$, the cost of the dot products with arbitrary $g$ would be reduced from six multiply-add pairs (three for each vector) to five (three for the first vector, then one to compute the dot products for each of the other two), even though there are more overall vector dot products to perform.
\end{remark}

\subsection{Finding an optimized computation}
Having defined complexity-reducing relations and given several
examples, we now show how they may be used to determine an optimized
evaluation of the stiffness matrix.  We shall work in the context of a
single complexity-reducing relation $\rho$, which may be composite.

In order to compute $\{(y^i)^t g\}_{i=1}^n$, we would like to pick
some $y^i \in Y$ and compute $(y^i)^t g$.  Then, we want to pick some
$y^j$ that is very close to $y^i$ under $\rho$ and then compute
$(y^j)^t g$.  Then, we pick some $y^k$ close to either
$y^i$ or $y^j$ and compute that dot product.  So, this process
continues until all the dot products have been computed.  Moreover,
since the vectors $Y$ depend only on the variational form and finite
element space and not the mesh or parameters, it is possible to do
this search once offline and generate low-level code that will exploit
these relations.  We first formalize the notion of finding the optimal
computation and then how the code may be generated.

We introduce a weighted, undirected graph $G=(Y,E)$ where $Y$ is our
collection of vectors defined above.  Our graph is completely
connected; that is, every pair of vectors $y^i,y^j$ are connected by
an edge.  The weight of this edge is defined to be $\rho(y^i,y^j)$.
We may think of walking along the edge from $y^i$ to $y^j$ as using
the dot product $(y^i)^t g$ to compute $(y^j)^t g$.  If $\rho$ is
composite, it will be helpful to know later which underlying relation
gave the minimum value used for $\rho$.  So, suppose that $\rho(y,z) =
\min \{ \rho_i(y,z) \}_{i=1}^R$.  For any fixed $y,z$, let
$\varrho(y,z)$ be a in integer in $[1,R]$ such that
$\rho_{\varrho(y,z)}(y,z) = \rho(y,z)$.  In addition to weights, we
thus associate with each edge $\{ y^i,y^j \}$ the entity
$\varrho(y^i,y^j)$.

A standard graph-theoretic object called a {\em minimum spanning tree}
is exactly what we need~\cite{Lev03}.  A spanning tree, which we shall
denote $(Y,E^\prime)$ is a subgraph that satisfies certain properties.
First, it contains all of the $n$ nodes of the original graph.
Second, $(Y,E^\prime)$ is connected.  Finally, $E^\prime$ has $n-1$
edges, so that there are no cycles (thus it is a tree).  Now, there
are possibly many spanning trees for a given graph.  Every spanning
tree has a {\em weight} associated with it that is the sum of the
weights of its edges.  A {\em minimum} spanning tree is a spanning
tree such that the sum of the edge weights is as small as possible.
Minimum spanning tree algorithms start with a particular node of the
graph, called the {\em root}.  Regardless of which root is chosen,
minimum spanning tree algorithms will generate trees with exactly the
same sum of edge weights.

While technically a minimum spanning tree is undirected, we can think
of it as being a directed graph with all edges going away from the
root.  Such a notion tells us how to compute all of the dot products
with minimal operations with respect to $\rho$.  We start with the
root node, which we assume is $y^0$, and compute $(y^0)^t g$.  Then,
for each of the children of $y^0$ in the tree, we compute the dot
products with $g$ using the result of $(y^0)^t g$.  Then, we use the
dot products of the children to compute the dot products of each of
the children's children, and so on. This is just a breadth-first
traversal of the tree. A depth-first traversal of the tree would also generate a correct algorithm, but it would likely require more motion of the computed results in and out of registers at run-time.

The total number of multiply-add pairs in this process is $m$ for
computing $(y^0)^t g$ plus the sum of the edge weights of
$(Y,E^\prime)$.  As the sum of edge weights is as small as it can be,
we have a minimum-cost algorithm for computing $\{(y^i)^t g\}_{i=1}^n$
with respect to $\rho$ for any $g \in \R^m$.  On the other hand, it is
not a true optimal cost as one could find a better $\rho$ or else use
relations between more than two vectors (say three coplanar vectors).
One other variation is in the choice of root vector.  If, for
example some $y^i$ has several elements that are zero, then it can be
dotted with $g$ with fewer than $m$ multiply add pairs.  Hence, we
pick some $\bar{y} \in Y$ such that the number of nonzero entries is
minimal to be the root.  We summarize these results in a theorem:

\begin{theorem}
Let $G=(Y,E)$ be defined as above and let $g \in \R^m$ be arbitrary.  The
total number of multiply-add pairs needed to compute $\{ (y^i)^t g\}_{i=1}^n$
is no greater than $m^\prime+w$, where $m^\prime$ is the minimum
number of nonzero entries of a member of $Y$ and $w$ is the weight of
a minimum spanning tree of $G$.
\end{theorem}

The overhead of walking through the tree at runtime would likely
outweigh the benefits of reducing the floating point cost.  We can
instead traverse the tree and generate low-level code for computing
all of the dot products - this function takes as an argument the
vector $g$ and computes all of the dot products of $Y$ with $g$.  An
example of such code was presented in~\cite{logg:article:07}.

\subsection{Comparison to spectral elements}
Our approach is remarkably different than the spectral element method.
In spectral element methods, one typically works with tensor products of
Lagrange polynomials over logically rectangular domains.  Efficient
algorithms for evaluating the stiffness matrix or its action follow
naturally by working dimension-by-dimension.  While such
decompositions are possible for unstructured shapes~\cite{KarShe99},
these are restricted to specialized polynomial bases.  On the other
hand, our approach is blind both to the element shape and kind of
approximating spaces used.  While spectral element techniques may
ultimately prove more effective when available, our approach will
enable some level of optimization in more general cases, such as
Raviart-Thomas-Nedelec~\cite{RavTho77a,RavTho77b,Ned80,Ned86} elements
on tetrahedra.

%% file: sections/experimental.tex
\section{Experimental results}
\label{experiment}
Here, we show that this optimization technique is successful at
generating low-flop algorithms for computing the element stiffness
matrices associated with some standard variational forms.  First, we
consider the bilinear forms for the Laplacian and advection in one
coordinate direction for tetrahedra.  Second, we study the trilinear
form for the weighted Laplacian.  In all cases, we generated the
element tensors using FFC, the FEniCS Form
Compiler~\cite{logg:submit:04,logg:www:04}, which in turn relies on
FIAT~\cite{Kir04,Kir05} to generate the finite element basis functions
and integration rules.  Throughout this section, we let $d=2,3$ refer
to the spatial dimension of $\Omega$.

\subsection{Laplacian}
We consider first the standard Laplacian operator
\begin{equation}
a(v,u) = \int_\Omega \nabla v(x) \cdot \nabla u(x) \dx.
\end{equation}
We gave a tensor product representation of the local stiffness matrix in
equation~(\ref{eq:laptensorrep}).  The indices of the local stiffness
matrix and the associated tensors are shown in
Table~\ref{table:Alaptriquadraticnosymm}.

\begin{table}
\caption{Element matrix indices and associated tensors (the slice
$A^0_{i\cdot}$ for each fixed index~$i$) displayed as vectors for the
Laplacian on triangles with quadratic basis functions.  All vectors
are scaled by six so they appear as integers.}
\begin{center}
\small
\begin{tabular}{ccc}
\begin{tabular}{|c|rrrr|} \hline
index & \multicolumn{4}{c|}{vector} \\ \hline\hline
(0, 0) & 3 & 3 & 3 & 3 \\ \hline
(0, 1) & 1 & 0 & 1 & 0 \\ \hline
(0, 2) & 0 & 1 & 0 & 1 \\ \hline
(0, 3) & 0 & 0 & 0 & 0 \\ \hline
(0, 4) & 0 & -4 & 0 & -4 \\ \hline
(0, 5) & -4 & 0 & -4 & 0 \\ \hline
(1, 0) & 1 & 1 & 0 & 0 \\ \hline
(1, 1) & 3 & 0 & 0 & 0 \\ \hline
(1, 2) & 0 & -1 & 0 & 0 \\ \hline
(1, 3) & 0 & 4 & 0 & 0 \\ \hline
(1, 4) & 0 & 0 & 0 & 0 \\ \hline
(1, 5) & -4 & -4 & 0 & 0 \\ \hline
\end{tabular}
&
\begin{tabular}{|c|rrrr|} \hline
index & \multicolumn{4}{c|}{vector} \\ \hline\hline
(2, 0) & 0 & 0 & 1 & 1 \\ \hline
(2, 1) & 0 & 0 & -1 & 0 \\ \hline
(2, 2) & 0 & 0 & 0 & 3 \\ \hline
(2, 3) & 0 & 0 & 4 & 0 \\ \hline
(2, 4) & 0 & 0 & -4 & -4 \\ \hline
(2, 5) & 0 & 0 & 0 & 0 \\ \hline
(3, 0) & 0 & 0 & 0 & 0 \\ \hline
(3, 1) & 0 & 0 & 4 & 0 \\ \hline
(3, 2) & 0 & 4 & 0 & 0 \\ \hline
(3, 3) & 8 & 4 & 4 & 8 \\ \hline
(3, 4) & -8 & -4 & -4 & 0 \\ \hline
(3, 5) & 0 & -4 & -4 & -8 \\ \hline
\end{tabular}
&
\begin{tabular}{|c|rrrr|} \hline
index & \multicolumn{4}{c|}{vector} \\ \hline\hline
(4, 0) & 0 & 0 & -4 & -4 \\ \hline
(4, 1) & 0 & 0 & 0 & 0 \\ \hline
(4, 2) & 0 & -4 & 0 & -4 \\ \hline
(4, 3) & -8 & -4 & -4 & 0 \\ \hline
(4, 4) & 8 & 4 & 4 & 8 \\ \hline
(4, 5) & 0 & 4 & 4 & 0 \\ \hline
(5, 0) & -4 & -4 & 0 & 0 \\ \hline
(5, 1) & -4 & 0 & -4 & 0 \\ \hline
(5, 2) & 0 & 0 & 0 & 0 \\ \hline
(5, 3) & 0 & -4 & -4 & -8 \\ \hline
(5, 4) & 0 & 4 & 4 & 0 \\ \hline
(5, 5) & 8 & 4 & 4 & 8
\\ \hline
\end{tabular}
\end{tabular}
\end{center}
\label{table:Alaptriquadraticnosymm}
\end{table}

\begin{table}
\caption{Element matrix indices and associated tensors (the slice
$A^0_{i\cdot}$ for each fixed index~$i$) for the Laplacian on
triangles with quadratic basis functions after transformation to make
use of symmetry.}
\begin{center}
\small
\begin{tabular}{ccc}
\begin{tabular}{|c|rrr|} \hline
index & \multicolumn{3}{c|}{vector} \\ \hline\hline
(0, 0) & 3 & 6 & 3 \\ \hline
(0, 1) & 1 & 1 & 0 \\ \hline
(0, 2) & 0 & 1 & 1 \\ \hline
(0, 3) & 0 & 0 & 0 \\ \hline
(0, 4) & 0 & -4 & -4 \\ \hline
(0, 5) & -4 & -4 & 0 \\ \hline
(1, 1) & 3 & 0 & 0 \\ \hline
(1, 2) & 0 & -1 & 0 \\ \hline
(1, 3) & 0 & 4 & 0 \\ \hline
(1, 4) & 0 & 0 & 0 \\ \hline
(1, 5) & -4 & -4 & 0 \\ \hline
\end{tabular}
&
\begin{tabular}{|c|rrr|} \hline
index & \multicolumn{3}{c|}{vector} \\ \hline\hline
(2, 2) & 0 & 0 & 3 \\ \hline
(2, 3) & 0 & 4 & 0 \\ \hline
(2, 4) & 0 & -4 & -4 \\ \hline
(2, 5) & 0 & 0 & 0 \\ \hline
(3, 3) & 8 & 8 & 8 \\ \hline
(3, 4) & -8 & -8 & 0 \\ \hline
(3, 5) & 0 & -8 & -8 \\ \hline
(4, 4) & 8 & 8 & 8 \\ \hline
(4, 5) & 0 & 8 & 0 \\ \hline
(5, 5) & 8 & 8 & 8 \\ \hline
\end{tabular}
\end{tabular}
\end{center}
\label{table:Alaptriquadraticsymm}
\end{table}

Because the element stiffness matrix is symmetric, we only need to
build the triangular part.  Even without any optimization techniques,
this naturally leads from computing $|P|^2$ contractions at a cost of
$d^2$ multiply-add pairs each to ${|P| + 1 \choose 2}$ contractions,
where $|P|$ is the dimension of the polynomial space $P$.
Beyond this, symmetry opens up a further opportunity for optimization.
For every element $e$, $G_e$ is symmetric.  The equality of its
off-diagonal entries means that the contraction can be performed in
${d + 1 \choose 2}$ rather than $d^2$ entries.  This is readily
illustrated in the two-dimensional case.  We contract a symmetric $2
\times 2$ tensor $G$ with an arbitrary $2 \times 2$ tensor $K$:
\begin{equation}
\begin{split}
G : K & =
\left( \begin{array}{cc}
G_{11} & G_{12} \\
G_{12} & G_{22}
\end{array} \right)
:
\left( \begin{array}{cc}
K_{11} &  K_{12} \\
K_{21} & K_{22}
\end{array} \right) \\
& = G_{11} K_{11} + G_{12} (K_{12} + K_{21}) + G_{22} K_{22} \\
& = \tilde{G}^t \hat{K},
\end{split}
\end{equation}
where $\tilde{G}^t = (G_{11}, G_{12}, G_{22})$ and
$\hat{K}^t = (K_{11}, K_{12} + K_{21}, K_{22})$.

This simple calculation implies a linear transformation $\hat{\cdot}$
from $\R^{d \times d}$ into $\R^{{d + 1 \choose 2}}$ obtained by
taking the diagonal entries of the matrix together with the sum of the
off-diagonal entries, that may be applied to each reference tensor, together
with an associated mapping $\tilde{\cdot}$ on symmetric tensors that
just takes the symmetric part and casts it as a vector.  Hence, the
overall cost of computing an element stiffness matrix before
optimizations goes from $|P|^2 d^2$ to $ {{|P| + 1 \choose 2}}{{ d + 1
\choose 2} }$.

An interesting property of this transformation of the reference tensor
is that it is {\em contractive} for the complexity-reducing relations
we consider.  The Hamming distance between two items under
$\hat{\cdot}$ is bounded by the Hamming distance between the items.
More precisely, $\rho(\hat{y},\hat{z}) \leq \rho(y,z)$.  Furthermore,
if items are colinear before transformation, their images will be as
well.  So, for the optimizations we consider, we will not destroy
any dependencies.  Moreover, the transformation may introduce
additional dependencies.  For example, before applying the
transformation, entries (0,1) and (1,5) are not closely related by
Hamming distance or colinearity, as seen in
Table~\ref{table:Alaptriquadraticnosymm}.  However, after the
transformation, we see that the same items in
Table~\ref{table:Alaptriquadraticsymm} are colinear.  Other examples
can be found readily.

We optimized the evaluation of the Laplacian for Lagrange finite
elements of degrees one through three on triangles and tetrahedra
using a composite complexity-reducing relation with $H^+$, $H^-$, and
$c$ defined in Examples~\ref{ex:colinear} and~\ref{ex:hamming}.  We
performed the optimization both with and without symmetry. The
results are shown in Tables~\ref{table:lapnosymm}~and~\ref{table:lapsymm}.
Figure~\ref{figure:mst} shows a diagram
of the minimum spanning tree computed by our code for the Laplacian on
quadratic elements using symmetry.  Each node of the graph is labeled
with a pair $(i,j)$ indicating the matrix entry (the vectors
themselves are displayed in Table~\ref{table:Alaptriquadraticsymm}), and the
edges are labeled with the associated weights.  Simple inspection
reveals that the sum of the edge weights is 14, which when added to 3
to compute the dot product for the root node, agrees with the entry
for quadratics in Table~\ref{table:lapsymm}.

\begin{figure}
\centerline{\includegraphics[width=\textwidth]{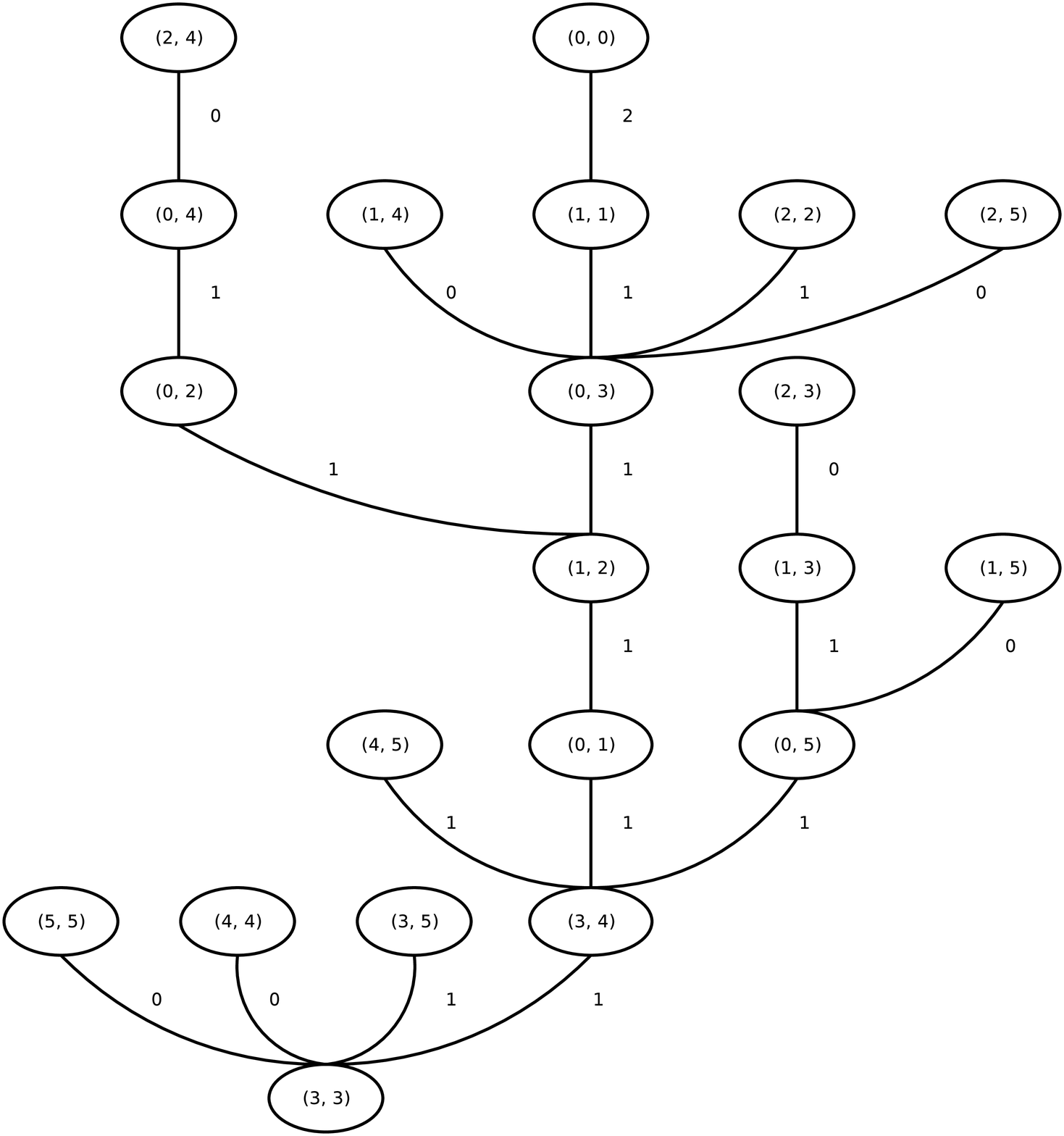}}
\caption{Minimum spanning tree for optimized computation of the Laplacian using quadratic elements on triangles.  The node labeled (3,3) is the root, and the flow is from bottom to top.}
\label{figure:mst}
\end{figure}

These techniques are successful in reducing the flop count, down to
less than one operation per entry on triangles and less than two on
tetrahedra for quadratic and cubic elements.  We showed
in~\cite{logg:article:07} for low degree elements on triangles that
going from standard numerical quadrature to tensor contractions led to
a significant reduction in actual run-time for matrix assembly.  From
the tensor contractions, we got another good speedup by simply
omitting multiplication by zeros.  From this, we only gained a modest
additional speedup by using our additional optimizations.  However,
this is most likely due to the relative costs of memory access and
floating point operations.  We have to load the geometric tensor from
memory and we have to write every entry of the matrix to memory.
So then, $n+m$ in the tables gives a lower bound on memory access for
computing the stiffness matrix.  Our optimizations lead to algorithms
for which there are a comparable number of arithmetic and memory
operations.  Hence, our optimization has succeeded in reducing the
cost of computing the local stiffness matrix to a small increment to
the cost of writing it to memory.

\begin{table}
\caption{Number of multiply-add pairs in the optimized algorithm for
computing the Laplacian element stiffness matrix on triangles and tetrahedra for
Lagrange polynomials of degree one through three without using
symmetry.}
\begin{center}
\begin{tabular}{c|c}
triangles & tetrahedra \\
\hline \\
\begin{tabular}{|c|ccc|c|}
\hline
degree & $n$ & $m$ & $nm$ & MAPs \\ \hline
1 & 9 & 4 & 36 & 13\\
2 & 36 & 4 & 144 & 25\\
3 & 100 & 4 & 400 & 74\\
\hline
\end{tabular}
&
\begin{tabular}{|c|ccc|c|}
\hline
degree & $n$ & $m$ & $nm$ & MAPs \\ \hline
1 & 16 & 9 & 144 & 43\\
2 & 100 & 9 & 900 & 205\\
3 & 400 & 9 & 3600 & 864 \\
\hline
\end{tabular}
\end{tabular}
\end{center}
\label{table:lapnosymm}
\end{table}

\begin{table}
\caption{Number of multiply-add pairs in the optimized algorithm for
computing the Laplacian element stiffness matrix on triangles and tetrahedra for
Lagrange polynomials of degree one through three using symmetry.}
\begin{center}
\begin{tabular}{c|c}
triangles & tetrahedra \\
\hline \\
\begin{tabular}{|c|ccc|c|}
\hline
degree & $n$ & $m$ & $nm$ & MAPs \\ \hline
1 & 6 & 3 & 18 & 9\\
2 & 21 & 3 & 63 & 17\\
3 & 55 & 3 & 165 & 46\\
\hline
\end{tabular}
&
\begin{tabular}{|c|ccc|c|}
\hline
degree & $n$ & $m$ & $nm$ & MAPs \\ \hline
1 & 10 & 6 & 60 & 27\\
2 & 55 & 6 & 330 & 101\\
3 & 210 & 6 & 1260 & 370\\
\hline
\end{tabular}
\end{tabular}
\end{center}
\label{table:lapsymm}
\end{table}

\subsection{Advection in one coordinate direction}
Now, we consider constant coefficient advection aligned with a
coordinate direction
\begin{equation}
a(v,u) = \int_{\Omega} v(x) \frac{\partial u(x)}{\partial x_1} \dx.
\end{equation}

This is part of the operator associated with constant coefficient
advection in some arbitrary direction --- optimizing the other
coordinate directions would give similar results.  These results are
shown in Table~\ref{table:xad}. Again, our optimization generates
algorithms for which the predominant cost of computing the element
matrix is writing it down, as there is significantly fewer than one
floating point cycle per matrix entry in every case.

\begin{table}
\caption{Number of multiply-add pairs in the optimized algorithm for
computing the coordinate-advection element stiffness matrix on
triangles and tetrahedra for Lagrange polynomials of degree one through three.}
\begin{center}
\begin{tabular}{c|c}
triangles & tetrahedra \\
\hline\\
\begin{tabular}{|c|ccc|c|}
\hline
degree & $n$ & $m$ & $nm$ & MAPs \\ \hline
1 & 9 & 2 & 18 & 4\\
2 & 36 & 2 & 72 & 22\\
3 & 100 & 2 & 200 & 59\\
\hline
\end{tabular}
&
\begin{tabular}{|c|ccc|c|}
\hline
degree & $n$ & $m$ & $nm$ & MAPs \\ \hline
1 & 16 & 3 & 48 & 9\\
2 & 100 & 3 & 300 & 35\\
3 & 400 & 3 & 1200 & 189\\
\hline
\end{tabular}
\end{tabular}
\end{center}
\label{table:xad}
\end{table}

\subsection{Weighted Laplacian}
Our final operator is the variable coefficient Laplacian:
\begin{equation}
a_w(v,u) = \int_{\Omega} w(x) \nabla v(x) \cdot \nabla u(x) \dx.
\end{equation}

This form may be viewed as a trilinear form $a(w,v,u)$ in which $w$ is
the projection of the coefficient into the finite element space.  For
many problems, this can be performed without a loss in order of
convergence.  For nonlinear problems, we will have to reassemble this
form at each nonlinear iteration, so it is an important step to
optimize.  The reference tensor is
\begin{equation}
A^0_{i\alpha} = \int_{E} \Phi_{\alpha_1}(X) \frac{\partial \Phi_{i_1}(X)}{\partial X_{\alpha_2}}
\frac{\partial \Phi_{i_2}(X)}{\partial X_{\alpha_3}} \dX
\end{equation}
and the geometric tensor is
\begin{equation}
  G_e^{\alpha} = w_{\alpha_1} \det F_e'
  \frac{\partial X_{\alpha_2}}{\partial x_{\beta}}
  \frac{\partial X_{\alpha_3}}{\partial x_{\beta}}.
\end{equation}
Note that the geometric tensor is the outer product of the geometric tensor
for the constant coefficient Laplacian, which we shall denote
$(G^L)_e$, with the coefficients of the weight function.  Unlike the
constant coefficient case, the amount of arithmetic per entry in the
element stiffness matrix grows with the polynomial degree.

We could simply proceed with the optimization as we did for other
forms --- the element tensor is just a collection of vectors that will
have to be dotted into $G_e$ for each element in the mesh, but now the
dot products are more expensive since the vectors are longer.  In this
case, $G_e$ must be explicitly formed for each element (this costs
$|P| d^2$ once $(G^L)_e$ is formed).  On the other hand, we could use
the decomposition of $G_e$ into $(G^L)_e$
and the coefficient vector $w_k$ and do the contractions in stages.
For example, we could organize the contraction as
\begin{equation}
  A_i^e = \left( A^0_{i,(\alpha_1,\alpha_2,\alpha_3)} \left( G^L \right)_e^{(\alpha_2,\alpha_3)} \right)
  w_{\alpha_1},
\end{equation}
that is, for each of the $|P|^2$ entries of the stiffness matrix, we
compute $|P|$ contractions of $d \times d$ tensors with $(G^L)_e$.
This is a similar optimization problem as the Laplacian, but with
$|P|$ times more vectors to optimize over.  After we do this set of
computations, we must compute $|P|^2$ dot products with the
coefficient vector $w_k$.  Note that the contractions with $(G^L)_e$
may be optimized, but the resulting vectors to dot with $w_k$ will not
be known until run-time and must be computed at full cost.  Hence, a
lower bound for this approach is $|P|$ multiply-add pairs per entry
(assuming that the contractions with $G^L$ were absolutely free).  On
the other hand, one could contract with the coefficient first to give
an array of $|P|^2$ tensors of size $d \times d$, and then contract
each of these with $(G^L)_e$.  As before, the first step can be
optimized, but the second step cannot.

In any of these cases, we may use the same transformations to exploit
symmetry as we did in the constant coefficient case.  Since
$G_e^{\alpha_1,\alpha_2,\alpha_3} = G_e^{\alpha_1,\alpha_3,\alpha_2}$,
we may view each slice $A^0_{i\cdot}$ of the reference tensor as an
array of $|P|$ tensors of size $d \times d$ and apply the
transformation to each of these.  If we fully form $G_e$, this reduces
the cost from $|P| d^2$ to $|P| {d+1 \choose 2}$.  In all of our
experiments, we made use of this.

In Table~\ref{table:wlapwhole}, we see the cost of computing the
weighted Laplacian by the first approach (optimizing directly the
tensor product $A_i^e = A_{i\alpha}^0 G_e^{\alpha}$). While the
optimizations are not as successful as for the constant coefficient
operators, we still get reductions of 30\%-50\% in the operation
counts.

When we perform the contraction in stages, we find more dependencies
(for example, the slices of two of the tensors could be colinear
although the entire tensors are not).  We show the cost of performing
the optimized stage for contracting with $(G^L)_e$ first in
Table~\ref{table:wlapgl} and for contracting with $w_k$ first in
Table~\ref{table:wlapcoeff}.

In order to get a fair comparison between these approaches, we must
factor in the additional costs of building $G_e$ or performing the
second stage of contraction.  Once $(G^L)_e$ is built and symmetrized,
it costs an additional $|P|{d+1 \choose 2}$ multiply-add pairs to
construct $G_e$.  If we optimize the computation of contracting with
$(G^L)_e$ first, we do not have to build $G_e$, but we must perform a
dot product with $w_k$ for each entry of the matrix.  This costs $|P|$
per contraction with ${|P|+1 \choose 2}$ entries in the matrix.  If we
optimize the contraction with each $w_k$ first, then we have an
additional ${|P|+1 \choose 2}$ contractions with $(G^L)_e$ at a cost
of ${d+1 \choose 2}$ each.  We expect that which of these will be most
effective must be determined (automatically) on a case-by-case
basis. Table~\ref{table:wlapcompare} shows the comparisons for the first
approach (labeled $G_e$), the second approach (labeled $(G^L)_e$) and
the third approach (labeled $w_k$) by indicating the cost of the
optimized computation plus the additional stages of computation.  In
most of these cases, contracting with the coefficient first leads to the
lowest total cost.

\begin{table}
\caption{Number of multiply-add pairs in the optimized algorithm for
computing the weighted Laplacian element stiffness matrix on
triangles and tetrahedra for Lagrange polynomials of degree one through three using
symmetry.}
\begin{center}
\begin{tabular}{c|c}
triangles & tetrahedra \\
\hline\\
\begin{tabular}{|c|ccc|c|}
\hline
degree & $n$ & $m$ & $nm$ & MAPs \\ \hline
1 & 6 & 9 & 54 & 27\\
2 & 21 & 18 & 378 & 218\\
3 & 55 & 30 & 1650 & 1110\\
\hline
\end{tabular}
&
\begin{tabular}{|c|ccc|c|}
\hline
degree & $n$ & $m$ & $nm$ & MAPs \\ \hline
1 & 10 & 24 & 240 & 108\\
2 & 55 & 60 & 3300 & 1650\\
3 & 210 & 120 & 25200 & 14334\\
\hline
\end{tabular}
\end{tabular}
\end{center}
\label{table:wlapwhole}
\end{table}

\begin{table}
\caption{Number of multiply-add pairs in the optimized algorithm for
performing all of the contractions with $(G^L)_e$ in the weighted
Laplacian on triangles and tetrahedra first, resulting in ${|P|+1 \choose 2}$ arrays
of length $|P|$ to contract with $w_k$.}
\begin{center}
\begin{tabular}{c|c}
triangles & tetrahedra \\
\hline\\
\begin{tabular}{|c|ccc|c|}
\hline
degree & $n$ & $m$ & $nm$ & MAPs \\ \hline
1 & 18 & 3 & 54 & 9\\
2 & 126 & 3 & 378 & 115\\
3 & 550 & 3 & 1650 & 683\\
\hline
\end{tabular}
&
\begin{tabular}{|c|ccc|c|}
\hline
degree & $n$ & $m$ & $nm$ & MAPs \\ \hline
1 & 40 & 6 & 240 & 27\\
2 & 550 & 6 & 3300 & 693\\
3 & 4200 & 6 & 25200 & 7021\\
\hline
\end{tabular}
\end{tabular}
\end{center}
\label{table:wlapgl}
\end{table}

\begin{table}
\caption{Number of multiply-add pairs in the optimized algorithm for
performing all of the contractions with the coefficient in the
weighted Laplacian on triangles and tetrahedra first, resulting in ${|P|+1 \choose
2}$ arrays of length ${d+1\choose 2} = 6$ to contract with $(G^L)_e$.}
\begin{center}
\begin{tabular}{c|c}
triangles & tetrahedra \\
\hline\\
\begin{tabular}{|c|ccc|c|}
\hline
degree & $n$ & $m$ & $nm$ & MAPs \\ \hline
1 & 18 & 3 & 54 & 7\\
2 & 63 & 6 & 378 & 138\\
3 & 165 & 10 & 1650 & 899\\
\hline
\end{tabular}
&
\begin{tabular}{|c|ccc|c|}
\hline
degree & $n$ & $m$ & $nm$ & MAPs \\ \hline
1 & 60 & 4 & 240 & 9\\
2 & 330 & 10 & 3300 & 465\\
3 & 1260 & 20 & 25200 & 7728\\
\hline
\end{tabular}
\end{tabular}
\end{center}
\label{table:wlapcoeff}
\end{table}

\begin{table}
\caption{Comparing the total number of multiply-add pairs for
fully forming $G_e$, contracting with $(G^L)_e$ first, and contracting
with $w_k$ first on both triangles and tetrahedra.}
\begin{center}
\begin{tabular}{c|c}
triangles & tetrahedra \\ \hline
\\
\begin{tabular}{|c|c|c|c|}\hline
\multicolumn{4}{|c|}{$G_e$} \\ \hline
degree & MST & additional & total \\ \hline
1 & 27 & 3*3 & 38 \\
2 & 218 & 3*6 & 236 \\
3 & 1110 & 3*10 & 1140 \\ \hline
\end{tabular}
&
\begin{tabular}{|c|c|c|c|}\hline
\multicolumn{4}{|c|}{$G_e$} \\ \hline
degree & MST & additional & total \\ \hline
1 & 108 & 6*4 & 132 \\
2 & 1650 & 6*10 & 1710 \\
3 & 14334  & 6*20 & 14454 \\ \hline
\end{tabular} \\
\\
\begin{tabular}{|c|c|c|c|}\hline
\multicolumn{4}{|c|}{$(G^L)_e$ first} \\ \hline
degree & MST & additional & total \\ \hline
1 & 9 & 6*3 & 27\\
2 & 115 & 21*6  & 241\\
3 & 683 & 55*10 & 1233 \\ \hline
\end{tabular}
&
\begin{tabular}{|c|c|c|c|}\hline
\multicolumn{4}{|c|}{$(G^L)_e$ first} \\ \hline
degree & MST & additional & total \\ \hline
1 & 27 & 10*4 & 67 \\
2 & 693 & 55*10 & 1234 \\
3 & \ 7021\ & 210*20 & 11221 \\ \hline
\end{tabular} \\
\\
\begin{tabular}{|c|c|c|c|}\hline
\multicolumn{4}{|c|}{$w_k$ first} \\ \hline
degree & MST & additional & total \\ \hline
1 & 7 & 6*3  & 25 \\
2 & 138 & 21*3 & 201 \\
3 & 899 & 55*3 & 1064 \\ \hline
\end{tabular}
&
\begin{tabular}{|c|c|c|c|}\hline
\multicolumn{4}{|c|}{$w_k$ first} \\ \hline
degree & MST & additional & total \\ \hline
1 & 9 & 10*6 & 69 \\
2 & 465 & 55*6 & 795\\
3 & \ 7728 \ & 210*6 & 8988 \\ \hline
\end{tabular}
\end{tabular}
\end{center}
\label{table:wlapcompare}
\end{table}

%% file: sections/search.tex
\section{Optimizing the optimization process}
\label{efficient}
Since our graph $(Y,E)$ is completely connected, we have
$|E|=O(|Y|^2)$ and our optimization process requires complexity that
is at least quadratic in the number of entries in the element
stiffness matrix.  In this section, we show how certain useful
complexity-reducing relations may be evaluated over all of $Y$ in
better than quadratic time, then discuss how we may build a sparse
graph $(Y,E^\prime)$ with $|E|=O(|Y|)$ that will admit a much more
efficient optimization process.  Even though this process must be run
only once per form and element (say the Laplacian with quadratics on
triangles), the quadratic algorithm can become very time consuming and
challenge a single computer's resources for forms of high arity using
high degree polynomials on tetrahedra.

\subsection{Search algorithms}
Before discussing how we may evaluate some of these
complexity-reducing relations over the collection $Y$ in better than
quadratic time, we first describe some basic notation we will use for
hash tables throughout this section and the next.

Hash tables are standard data structures~\cite{Lev03} that associate
each member of a set of {\em keys} to some value, possibly drawn from
another set of objects.  The important point about hash tables is that
the basic operations of setting and getting values are expected to be
independent of the number of entries in the table (expected constant
time access).  Many higher-level programming languages have library
support or built-in features supporting hash tables (many
implementations of the standard template library in C++ include hash
tables, and scripting languages such as Python and Perl have them
built in as primitive types).

We begin by establishing some notation for the basic operations we
use.  If $a$ is a key of table $T$, then we find the value associated
with $a$ by $T[a]$.  If there is no value associated with $a$ (that
is, if $a$ is not a key of $T$, we may update $T$ by adding a key $a$
associated with value $b$ by the notation $T[a] \leftarrow b$.  We use
the same notation to indicate setting a new value to an existing key.

As before, we label the vectors $y^i$ for $1 \leq i \leq n $.  We want
to partition the labels into a set of subsets $\mathcal E$ such that
for each $E \in \mathcal E$, the vectors associated with each label in
$E$ are equal.  Moreover, if two vectors are equal, then their labels
must belong to the same $E$.  This is easily accomplished by setting
up a hash table whose keys are vectors and whose values are subsets of
the integers $1 \leq i \leq n$.  This process is described in
Algorithm~\ref{alg:equalhash}
\begin{algorithm}
\begin{algorithmic}
\STATE $E$ an empty table mapping vectors to subsets of $\{i\}_{i=1}^n$.
\FORALL{$1 \leq i \leq n$}
\IF{ $y^i$ is a key of $E$ }
  \STATE $E[y^i] \leftarrow E[v^i] \cup \{ i \}$
\ELSIF{ $-y^i$ is a key of $E$ }
    \STATE $E[-y^i] \leftarrow E[-y^i] \cup \{ i \}$
\ELSE
  \STATE $E[y^i] \leftarrow \{ i \}$
\ENDIF
\ENDFOR
\end{algorithmic}
\caption{Determining equality among vectors}
\label{alg:equalhash}
\end{algorithm}

Floating point arithmetic presents a slight challenge to hashing.
Numbers which are close together (within some tolerance) that should
be treated as equal must be rounded to so that they are indeed equal.
Hashing relies on a function that maps items into a set of integers
(the "hash code").  These functions are discontinuous and sensitive to
small perturbations.  For most numerical algorithms in floating point
arithmetic, we may define equality to be "near equality", but hash
tables require us to round or use rational arithmetic before any
comparisons are made.  We have successfully implemented our algorithms
in both cases.

As an alternative to hashing, one could form a binary search tree or
sort the vectors by a lexicographic ordering.  These would rely on a
more standard "close to equal" comparison operation, but only run in
$O(mn \log{(mn)})$ time.  So, for large enough data sets, hashing will
be more efficient.

We may similarly partition the labels into a set of subsets $\mathcal
C$ such that for each $C \in \mathcal C$, the vectors associated with
the labels in $C$ are colinear.  Similarly, if two vectors are
colinear, then their labels must belong to the same $C$.  This process
may be performed by constructing the collection of unit vectors
$\hat{A}$, with $\hat{y}_i = \frac{y_i}{\| y_i \|}$ for each $1 \leq i
\leq n$ for some norm $\| \cdot \|$ and using
Algorithm~\ref{alg:equalhash} on $\hat{A}$.

Finding vectors that are close together in Hamming distance is more
subtle.  At worst, the cost is $O(mn^2)$, as we have to compare every
entry of every pair of vectors.  However, we may do this in expected
linear time with some assumptions about $Y$.  We first describe the
algorithm, then state the conditions under which the algorithm
performs in worse than linear time.

Our vectors each have $m$ components.  We start by forming $m$ empty
hash tables.  Each $H_i$ will map numbers that appear in the
$i^\mathrm{th}$ position of any vector to the labels of vectors that
have that entry in that position.  This is presented in
Algorithm~\ref{alg:hammingtable}, in which $y^i_j$ denotes the
$j^\mathrm{th}$ entry of $y^i$.

\begin{algorithm}
\begin{algorithmic}
\FORALL{$1 \leq i \leq d$}
  \STATE $H_i$ an empty table mapping numbers to
  sets of vector labels from $\{i\}_{i=1}^n$
\ENDFOR
\FORALL{$1 \leq i \leq n$}
  \FORALL{$1 \leq j \leq m$}
    \IF{$y^i_j$ is a key of $H_j$}
      \STATE $H_j[ y^i_j ] \leftarrow H_j[y^i_j] \cup \{ i \}$
    \ELSE
      \STATE $H_j[ y^i_j ] := \{ i \}$
    \ENDIF
  \ENDFOR
  \ENDFOR
\end{algorithmic}
\caption{Mapping unique entries at each position to vectors containing them}
\label{alg:hammingtable}
\end{algorithm}

This process runs in expected $O(nm)$ time.  From these tables, we can
construct a table that gives the Hamming distance between any two
vectors, as seen in Algorithm~\ref{alg:hammingdistance}.  This
algorithm counts down from $d$ each time it discovers an entry that
two vectors share.  Our algorithm reflects the symmetry of the Hamming
distance.
\begin{algorithm}
\begin{algorithmic}
\STATE $D$ an empty table
\FORALL {$1 \leq i \leq n$}
  \STATE $D[i]$ an empty table
\ENDFOR
\FORALL{$1 \leq i \leq m$}
  \FORALL{$a$ in the keys of $H_i$}
    \FORALL{unique combinations $k,\ell$ of labels in $H_i[a]$}
      \STATE $\alpha := \min(k,\ell), \beta := \max(k,\ell)$
      \IF{$D[\alpha]$ has a key $\beta$}
        \STATE $D[\alpha][\beta] \leftarrow D[\alpha][\beta] - 1$
      \ELSE
        \STATE $D[\alpha][\beta] := m-1$
      \ENDIF
    \ENDFOR
  \ENDFOR
\ENDFOR
\end{algorithmic}
\caption{Computing Hamming distances efficiently}
\label{alg:hammingdistance}
\end{algorithm}

On output, for $1 \leq i < j \leq n$, if $D[i][j]$ has no entry, the
distance between $v^i$ and $v_j$ is $m$.  Otherwise, the table
contains the Hamming distance between the two vectors.  

Regarding complexity, there is a double loop over the entries of each
$H_i$.  Hence, the algorithm is quadratic in the maximum number of
vectors that share the same entry at the same location.  Presumably,
this is considerably less than $n$ on most data sets.

\subsection{Using a sparse graph}
If we create a graph $(Y,\tilde{E})$ with significantly fewer edges than $(Y,E)$, we may be able to get most of the reduction in operation count while having a more efficient optimization process.  For example, we might choose to put a cutoff on $\rho$, only using edges that have a large enough complexity-reduction.  So, we can define the set of edges being
\begin{equation}
E_k = \{ \{ y^i , y^j \} : \rho(y^i,y^j) \leq k \}
\end{equation}

For example, if we use Algorithms~\ref{alg:hammingtable}
and~\ref{alg:hammingdistance} to evaluate $H^+$ over all pairs from
$Y$, then we are using $\rho = H^+$ and $k=m-1$.  Also, note that our
structure $D$ encodes a sparse graph.  That is, the vectors of $Y$ are
the nodes of the graph, the neighbors of each $y^i$ are simply the
keys of $D[y^i]$, and the edge weight between some $y^i$ and neighbor
$y^j$ is $D[y^i][j^j]$. It is not hard to move from this table into
whatever data structure is used for graphs.

Then, we could add colinearity or other complexity-reducing relations
to this graph.  If we use Algorithm~\ref{alg:equalhash} on the unit
vectors to determine sets of colinear vectors, we can update the the
graph by either adding edges or updating the edge weights for each
pair of colinear vectors.

If $|E_k| = O(|Y|)$, then computing a minimum spanning tree will
require only $O(n \log n)$ time rather than $O(n^2 \log n)$.  However,
there is no guarantee that $(Y,E_k)$ will be a connected graph.  Some
vectors might not have close neighbors, or else some subgraphs do not
connect with each other.  An optimized algorithm can still be obtained
by finding the connected components of the $(Y,E_k)$ and finding a
minimum spanning tree for each component.  Then, the total cost of the
computation is $m$ times the number of connected components plus the
sum of the weights of each of the minimum spanning trees.

%% file: sections/conclusion.tex
\section{Conclusion and ongoing work}
\label{conc}
We have developed a general optimization strategy for the evaluation
of local stiffness matrices for finite elements.  This is based on
first formulating the computation as a sequence of tensor
contractions, then introducing a new concept of complexity-reducing
relations that allows us to set the optimization in a graph context.
The optimization itself proceeds by computing a minimum spanning tree.
These techniques worked very well at reducing the cost of evaluating
finite element matrices for several forms using Lagrange elements of
degrees one through three on triangles and tetrahedra.  Finally, we
discussed some efficient algorithms for detecting equality and
colinearity and for evaluating the pairwise Hamming distance over the
entire set of tensors.

We hope to extend our optimizations in several ways.
First, we remarked that it may be possible to enrich the collection of tensors
in such as way as to deliever a minimum spanning tree with lower
total weight.  We hope to explore such a process in the future.  Also,
we saw in~\cite{logg:article:07} that frequently, some of the tensors
are linear combinations of two or more other tensors.  However,
both locating and making use of such relations in a more formal
context has been difficult.  We are working on geometric search
algorithms to locate linear dependencies efficiently.  However, once
they are located, our optimization process must occur over a
hypergraph rather than a graph.  Finding a true minimum is also much
more difficult, and we are working on heuristics that will allow us to
combine these two approaches.

Finally, we plan to integrate our optimization strategy with FFC.
While FFC currently generates very efficient code for evaluating
variational forms, we will improve upon this generated code by piping
the tensors through our optimization process before generating code to
perform the contractions.  This will lead to a domain-specific
optimizing compiler for finite elements; by exploiting latent
mathematical structure, we will automatically generate more efficient
algorithms for finite elements than people write by hand.

%% file: paper.bbl
\begin{thebibliography}{10}

\bibitem{www:Analysa}
{\sc B.~Bagheri and R.~Scott}, {\em Analysa}.
\newblock \url{http://people.cs.uchicago.edu/~ridg/al/aa.html}.

\bibitem{www:deal}
{\sc W.~Bangerth, R.~Hartmann, and G.~Kanschat}, {\em {\tt deal.{I}{I}}
  {D}ifferential {E}quations {A}nalysis {L}ibrary}.
\newblock \url{http://www.dealii.org}.

\bibitem{BreSco94}
{\sc S.~C. Brenner and L.~R. Scott}, {\em The Mathematical Theory of Finite
  Element Methods}, Springer-Verlag, 1994.

\bibitem{Cia76}
{\sc P.~G. Ciarlet}, {\em Numerical Analysis of the Finite Element Method}, Les
  Presses de l'Universite de Montreal, 1976.

\bibitem{www:GetDP}
{\sc P.~Dular and C.~Geuzaine}, {\em Get{DP}: a {G}eneral environment for the
  treatment of {D}iscrete {P}roblems}.
\newblock \url{http://www.geuz.org/getdp/}.

\bibitem{logg:www:01}
{\sc J.~Hoffman, J.~Jansson, and A.~Logg}, {\em {DOLFIN}}.
\newblock \url{http://www.fenics.org/dolfin/}.

\bibitem{logg:preprint:06}
{\sc J.~Hoffman and A.~Logg}, {\em {DOLFIN}: {D}ynamic {O}bject oriented
  {L}ibrary for {FIN}ite element computation}, Tech. Report 2002--06, Chalmers
  Finite Element Center Preprint Series, 2002.

\bibitem{KarShe99}
{\sc G.~E. Karniadakis and S.~J. Sherwin}, {\em Spectral/{$hp$} element methods
  for {CFD}}, Numerical Mathematics and Scientific Computation, Oxford
  University Press, New York, 1999.

\bibitem{Kir04}
{\sc R.~C. Kirby}, {\em {FIAT}: A new paradigm for computing finite element
  basis functions}, ACM Trans. Math. Software, 30 (2004), pp.~502--516.

\bibitem{Kir05}
\leavevmode\vrule height 2pt depth -1.6pt width 23pt, {\em Optimizing {FIAT}
  with the level 3 {BLAS}}, submitted to ACM Trans. Math. Software,  (2005).

\bibitem{logg:article:07}
{\sc R.~C. Kirby, M.~Knepley, A.~Logg, and L.~R. Scott}, {\em Optimizing the
  evaluation of finite element matrices}, To appear in SIAM J. Sci. Comput.,
  (2005).

\bibitem{KirKne05}
{\sc R.~C. Kirby, M.~Knepley, and L.~R. Scott}, {\em Evaluation of the action
  of finite element operators}, submitted to BIT,  (2005).

\bibitem{logg:submit:04}
{\sc R.~C. Kirby and A.~Logg}, {\em A compiler for variational forms}.
\newblock submitted to {ACM} Trans. Math. Softw., 2005.

\bibitem{Lev03}
{\sc A.~Levitin}, {\em Introduction to the Design and Analysis of Algorithms},
  Addison-Wesley, 2003.

\bibitem{logg:www:04}
{\sc A.~Logg}, {\em The {FE}ni{CS} {F}orm {C}ompiler {FFC}}.
\newblock \url{http://www.fenics.org/ffc/}.

\bibitem{www:Sundance}
{\sc K.~Long}, {\em Sundance}.
\newblock \url{http://csmr.ca.sandia.gov/~krlong/sundance.html}.

\bibitem{Lon03}
\leavevmode\vrule height 2pt depth -1.6pt width 23pt, {\em Sundance, a rapid
  prototyping tool for parallel {PDE}-constrained optimization}, in Large-Scale
  PDE-Constrained Optimization, Lecture notes in computational science and
  engineering, Springer-Verlag, 2003.

\bibitem{Ned80}
{\sc J.-C. N{\'e}d{\'e}lec}, {\em Mixed finite elements in {${\bf R}\sp{3}$}},
  Numer. Math., 35 (1980), pp.~315--341.

\bibitem{Ned86}
\leavevmode\vrule height 2pt depth -1.6pt width 23pt, {\em A new family of
  mixed finite elements in {${\bf R}\sp 3$}}, Numer. Math., 50 (1986),
  pp.~57--81.

\bibitem{www:FreeFEM}
{\sc O.~Pironneau, F.~Hecht, and A.~L. Hyaric}, {\em Free{FEM}}.
\newblock \url{http://www.freefem.org/}.

\bibitem{RavTho77a}
{\sc P.-A. Raviart and J.~M. Thomas}, {\em A mixed finite element method for
  2nd order elliptic problems}, in Mathematical aspects of finite element
  methods (Proc. Conf., Consiglio Naz. delle Ricerche (C.N.R.), Rome, 1975),
  Springer, Berlin, 1977, pp.~292--315. Lecture Notes in Math., Vol. 606.

\bibitem{RavTho77b}
\leavevmode\vrule height 2pt depth -1.6pt width 23pt, {\em Primal hybrid finite
  element methods for {$2$}nd order elliptic equations}, Math. Comp., 31
  (1977), pp.~391--413.

\end{thebibliography}
